\documentclass[a4paper,11pt]{article}

\usepackage[utf8]{inputenc}
\usepackage[english]{babel}
\usepackage[T1]{fontenc}
\usepackage{lmodern}
\usepackage{authblk}
\usepackage[margin=1.1in]{geometry}
\usepackage{amsmath,amssymb,amsthm}	
\usepackage{color}			
\usepackage{graphicx}
\usepackage{float}
\usepackage{textcomp}
\usepackage{amsmath, amsthm, amssymb}
\usepackage{enumitem}
\usepackage{fancyhdr}	
\usepackage{hyperref}
\usepackage{tabularx}
\usepackage{environ}

\definecolor{mountain}{RGB}{16, 172, 132}  
\definecolor{joute}{RGB}{46, 134, 222}  

\hypersetup{
	colorlinks=true,
	linkcolor=joute,
	filecolor=magenta,
	urlcolor=cyan,
	citecolor=mountain
}
\makeatletter
\newcommand{\problemtitle}[1]{\gdef\@problemtitle{#1}}
\newcommand{\probleminput}[1]{\gdef\@probleminput{#1}}
\newcommand{\problemquestion}[1]{\gdef\@problemquestion{#1}}
\NewEnviron{decproblem}{
	\problemtitle{}\probleminput{}\problemquestion{}
	\BODY
	\par\addvspace{.5\baselineskip}
	\noindent
	\begin{tabularx}{\textwidth}{@{\hspace{\parindent}} l X c}
		\multicolumn{2}{@{\hspace{\parindent}}l}{\@problemtitle} \\
		\textbf{Input:} & \@probleminput \\
		\textbf{Question:} & \@problemquestion
	\end{tabularx}
	\par\addvspace{.5\baselineskip}}
\NewEnviron{genproblem}{
	\problemtitle{}\probleminput{}\problemquestion{}
	\BODY
	\par\addvspace{.5\baselineskip}
	\noindent
	\begin{tabularx}{\textwidth}{@{\hspace{\parindent}} l X c}
		\multicolumn{2}{@{\hspace{\parindent}}l}{\@problemtitle} \\
		\textbf{Input:} & \@probleminput \\
		\textbf{Output:} & \@problemquestion
	\end{tabularx}
	\par\addvspace{.5\baselineskip}}

\newtheorem{theorem}{Theorem}[section]
\newtheorem{lemma}{Lemma}
\newtheorem{proposition}{Proposition}
\newtheorem{definition}{Definition}
\newtheorem{corollary}{Corollary}

\newtheorem{remark}{Remark}
\newtheorem{example}{Example}

\newcommand{\A}{\mathcal{A}}  
\newcommand{\B}{\mathcal{B}}  
\newcommand{\R}{\mathcal{R}}  
\newcommand{\AF}{F}  
\newcommand{\SD}{\mathrm{SD}}  
\newcommand{\PREF}{\mathrm{PREF}}  
\newcommand{\NAIV}{\mathrm{NAIV}}  
\newcommand{\ADM}{\mathrm{ADM}}  
\newcommand{\CF}{\mathrm{CF}}  
\newcommand{\IRR}{\mathrm{IRR}}  
\newcommand{\I}{\mathcal{I}}  
\newcommand{\F}{\mathcal{F}}  
\renewcommand{\S}{\mathcal{S}}  
\newcommand{\IS}{\Sigma}  
\newcommand{\card}[1]{\vert #1 \vert}  
\newcommand{\poly}{\csf{poly}}  
\newcommand{\imp}{\rightarrow}  

\newcommand{\cc}[1]{\mathcal{#1}}  
\newcommand{\csf}[1]{\textsf{#1}}  
\newcommand{\csmc}[1]{\textsc{#1}}  
\newcommand{\ctt}[1]{\texttt{#1}}  

\title{On the preferred extensions of argumentation frameworks: bijections with naive 
sets}
\author[1]{Mohammed Elaroussi}
\author[2]{Lhouari Nourine}
\author[1]{Mohammed Said Radjef}
\author[2, 3]{\\ Simon Vilmin}

\affil[1]{Unit\'e de Recherche LaMOS. Facult\'e des Sciences Exactes. Universit\'e de 
Bejaia, 06000 Bejaia, Algeria}
\affil[2]{Universit\'e Clermont-Auvergne, CNRS, Mines de Saint-\'Etienne, 
	Clermont-Auvergne-INP, LIMOS, 63000 Clermont-Ferrand, France.}
\affil[3]{Univ Lyon, INSA Lyon, CNRS, UCBL, Centrale Lyon, Univ Lyon 2, LIRIS,
UMR5205, F-69621 Villeurbanne, France.}

\begin{document}
	\maketitle	
	
\begin{abstract}
This paper deals with the problem of finding the preferred extensions of an argumentation 
framework by means of a bijection with the naive sets of another framework.
First, we consider the case where an argumentation framework is naive-bijective: its 
naive sets and preferred extensions are equal.
Recognizing naive-bijective argumentation frameworks is hard, but we show that 
it is tractable for frameworks with bounded in-degree.
Next, we give a bijection between the preferred extensions of an argumentation framework 
being admissible-closed (the intersection of two admissible sets is admissible) 
and the naive sets of another framework on the same set of arguments.
On the other hand, we prove that identifying admissible-closed argumentation frameworks 
is \csf{coNP}-complete.
At last, we introduce the notion of irreducible self-defending sets as those that are not 
the union of others. 
It turns out there exists a bijection between the preferred extensions of an 
argumentation framework and the naive sets of a framework on its irreducible 
self-defending sets.
Consequently, the preferred extensions of argumentation frameworks with some lattice 
properties can be listed with polynomial delay and polynomial space.

\end{abstract}

\section{Introduction}
\label{sec:introduction}

Abstract argumentation frameworks \cite{dung1995acceptability} have become a major 
research trend within the field of 
artificial intelligence \cite{bench2007argumentation,rahwan2009argumentation}. 
They also provide tools to study other areas such as decision support systems 
\cite{amgoud2009using}, machine learning \cite{cocarascu2016argumentation}, and agent 
interaction in multi-agent systems \cite{mcburney2009dialogue}.  
In particular, Dung \cite{dung1995acceptability} shows how $n$-person games 
\cite{von1947theory}, logic default reasoning \cite{pollock1987defeasible} and logic 
programming \cite{apt1990logic} can be seen as instances of abstract argumentation.

Informally, an abstract argumentation framework consists of a set 
of arguments together with a binary relation representing attacks between arguments.
The aim of an argument framework is to evaluate the acceptability (or the relevance) of 
groups of arguments.
On this purpose, Dung \cite{dung1995acceptability} introduces \emph{admissible sets}.
An admissible set is a collection of arguments that are not attacking each other 
(conflict-free set) and that are counterattacking all externals attackers 
(self-defending set).
Then, Dung distinguishes several types of admissible sets called \emph{extensions}, which 
express different perspectives on arguments that are collectively acceptable.
Examples of extensions are preferred, stable, complete, or grounded extensions.
In particular, preferred extensions represent inclusion-wise maximal admissible sets, and 
hence the largest (\textit{w.r.t} set-inclusion) acceptable collections of arguments.
We illustrate these notions with the help of an example of Baroni et al. 
\cite{baroni2020acceptability} about vegetarianism.
\begin{example}
Let us consider the next four arguments:
\begin{itemize} \itemsep-0.2em
\item[(A)] \textit{``a vegetarian diet should be adopted because it is generally 
healthier''};
\item[(B)] \textit{``a vegetarian diet should not be adopted because it can promote 
specific health problems''};
\item[(C)] \textit{``we should not eat animals because animals have rights because they 
are sentient beings''};
\item[(D)] \textit{``animals cannot possess rights because they have no moral judgment''}.
\end{itemize}
Following Baroni et al. \cite{baroni2020acceptability}, the arguments $A$ and $B$ attack 
each other as well as the arguments $C$ and $D$, and $C$ attacks $B$.

The set $\{D, B\}$ is admissible: neither $D$ nor $B$ attacks the other (or 
itself) and $D$ defends $B$ against $C$.
In fact, $\{D, B\}$ is preferred as adding either $A$ or $C$ would break the 
conflict-free set property. 
On the other hand, $\{B\}$ is not admissible since there is no defense against $C$. 
The set $\{C,D\}$ is not admissible either as $C$ attacks $D$.
\end{example}

Naturally, extensions are at the core of numerous decision and search 
problems revolving around argumentation frameworks.
Unfortunately, most of these problems are hard to solve for several argumentation 
extensions \cite{dvovrak2017computational, gaggl2020design, kroll2017complexity}.
Nevertheless, these hardness results have motivated the search for tractable classes of 
argumentation frameworks among which we quote the frameworks being without even cycles 
\cite{dunne2001complexity}, symmetric \cite{coste2005symmetric} or 
bipartite \cite{dunne2007computational}.

In this paper, we are interested in the enumeration of the preferred extensions 
of an argumentation framework, being a hard problem \cite{dimopoulos1996graph}.
Even though preferred extensions are only as informative as (maximal) 
conflic-free sets within the context of logic-based frameworks 
\cite{caminada2011limitations, amgoud2012stable}, their enumeration
remains one of the most general problem in argumentation frameworks.
Indeed, several recent works contributed to this task as well as the enumeration of other 
extensions in argumentation frameworks \cite{ alfano2019scaling,  
cerutti2018impact, charwat2015methods, elaroussi:hal-03184819}. 
Enumerating preferred extensions provides complete information concerning the 
justification status of 
arguments and allows to easily resolve other significant argumentation frameworks 
problems.

In this context, our strategy draws inspiration from the work of Dunne \textit{et al.}
\cite{dunne2015characteristics} on realizability and recasting problems for argumentation 
frameworks.
More precisely, we seek to express the set of preferred extensions of a given framework 
using a bijective correspondence with the inclusion-wise maximal conflict-free 
sets, called the naive sets, of another argumentation 
framework. 
For example, symmetric argumentation frameworks admit a bijection between their 
naive sets and 
preferred extensions \cite{coste2005symmetric}.
Since it is possible to list the naive sets of an argumentation framework with 
polynomial delay and space \cite{johnson1988generating}, our approach can be used to 
identify new classes of frameworks where the preferred extensions can be enumerated 
efficiently.
On this purpose, we obtain the following contributions:

\paragraph{1. Naive-bijective argumentation frameworks.} (Section 
\ref{sec:naive-bijective }).
An argumentation framework is \emph{naive-bijective} if its naive sets are 
equal 
to its 
preferred extensions.
For instance, symmetric argumentation frameworks are naive-bijective  
\cite{coste2005symmetric}. This is also the case of every logic-based argumentation framework, as 
demonstrated for example by Amgoud \cite{amgoud2012stable}.
Nevertheless, the class of naive-bijective argumentation frameworks is not very well 
characterized.
We give an exact characterization of naive-bijective  argumentation frameworks.
We deduce an algorithm to recognize this class of argumentation frameworks, running 
in 
polynomial time in frameworks with bounded in-degree. 
In contrast, we prove that this recognition problem is \csf{coNP}-complete in general.

\paragraph{2. Naive-recasting and admissible-closed argumentation frameworks.} 
(Section \ref{sec:admissible-closed}).
An argumentation framework is \emph{naive-recasting} if its preferred extensions are 
the 
naive sets of an other argumentation framework.
This class of argumentation framework has been introduced by Dunne \textit{et al.} \cite{dunne2015characteristics}, and studied in a series of works \cite 
{baumann2014compact,dunne2016investigating, puhrer2020realizability}. 
In particular, the authors in \cite{dunne2015characteristics} show that deciding 
whether a given argumentation framework is naive-recasting is \csf{coNP}-hard.
In our contribution, we introduce \emph{admissible-closed} argumentation frameworks.
An argumentation framework belongs to this class if its admissible sets are 
closed 
under intersection.
Using implicational systems \cite{elaroussi:hal-03184819, wild2017joy}, we 
show that admissible-closed argumentation frameworks are naive-recasting.
Eventually, we prove that recognizing admissible-closed argumentation frameworks is a 
\csf{coNP}-complete problem.

\paragraph{3. Irreducible self-defending sets.} (Section \ref{sec:irreducible}).
Using the lattice structure of self-defending sets \cite{dung1995acceptability, gratzer2011lattice} we introduce irreducible self-defending sets.
A self-defending set is irreducible if it cannot be obtained as the union of other 
self-defending sets.
We construct a bijection between the preferred extensions of an argumentation 
framework 
and the naive sets of another framework on its irreducible self-defending sets.
We deduce that when an argumentation framework meets some lattice properties, it is 
possible to enumerate its preferred extensions with polynomial delay and polynomial 
space.

\section{Background} \label{sec:preliminaries}

All the objects considered in this paper are finite. Let $\A$ be a set.
We denote by $2^{\A}$ the powerset of $\A$.
Let $S \subseteq \A$.
For convenience, and mostly in examples and figures, we may write $S$ as the 
concatenation of its elements.
That is, if $S = \{x_1, \dots, x_n\}$, we may write it $x_1 \dots x_n$.
The \emph{complement} $\bar{S}$ of $S$ in $\A$ is given by $\bar{S} = \A \setminus S$.
Let $\F \subseteq 2^{\A}$ be a set system over $\A$.
The complementary set of $\F$ (with respect to $\A$) is the set system  $\bar{\F}$ 
defined  by $\bar{\F} = \{\bar{S} \mid S \in \F\}$.

\paragraph{Argumentation frameworks.}
We follow \cite{dung1995acceptability}.
An \emph{argumentation framework} is
a pair $\AF = \langle \A, \R \rangle$ where $\A$ is a set of arguments, and $\R$ 
(attacks) is a binary relation on $\A$, i.e. $\R \subseteq  \A\times \A$.
A pair $(x, y$) in $\R$ depicts an attack of the argument $x$ against the argument $y$.
We can represent an argumentation framework $\AF$ by a directed graph $G = (\A, 
\R) $, known as the \emph{attack graph} of $\AF$. 
Given a subset $S$ of arguments, we denote by $S^+$ the set of arguments attacked by 
$S$, i.e. $S^+ = \{x \in \A \mid \exists y \in S, (y,x) \in \R\}$. 
Likewise, the set of arguments	that attack at least one argument of $S$ is denoted by 
$S^-$, that is $S^- = \{ x \in \A \mid \exists y\in S, (x,y) \in \R \}$.
In case $S$ is a singleton $\{y\}$, we write $y^-$ as a shortcut for $\{y\}^-$.
Finally, we put $\Gamma(S) = S^+ \cup S^-$.

Given a set of arguments S of $\A$, an argument $x \in 
\A$ is \emph{acceptable} with respect to $S$ if for each $y \in 
\A$, $(y, x) \in \R$ implies that $(z, y) \in \R$ for some $z$ in $S$.
We say that $S$ is \emph{conflict-free} (an \emph{independent set} in graph theory's 
terminology) if $S \cap S^+ = \emptyset$.
We denote by $\CF(\AF)$ the family of all conflict-free sets of $\AF$.
We call $\S(\AF)$ the set of all self-attacking arguments, that is $\S(\AF) = \{x \in \A 
\mid 
(x, x) \in \R \}$.
The inclusion-wise maximal conflict-free subsets of $\A$  are called \emph{naive 
sets} \cite{bondarenko1997abstract} ({\em maximal independent set} in graph theoretic 
terms). 
The family of all naive sets of $\AF$ is denoted by $\NAIV(\AF)$.
A set $S\subseteq \A$ is  \emph{self-defending} if  $S^-\subseteq S^+$.
The family of all self-defending sets of $\AF$ is denoted by 
$\SD(\AF)$. 
Using conflict-free and acceptable sets, Dung \cite{dung1995acceptability} 
defines several extensions. 
A conflict-free set $S\subseteq A$ is \emph{admissible} if 
$S^-\subseteq S^+$.
The family of all admissible sets of $\AF$ is  denoted by
$\ADM(\AF)$.
The inclusion-wise maximal admissible subsets of $\A$ are called \emph{preferred} 
extensions, and the family of all preferred extensions of $\AF$ is 
denoted by $\PREF(\AF)$. 
We refer the reader to Baroni \textit{et al.}~\cite{baroni2011introduction} 
for more discussion on semantics.

\begin{example}[Running example] \label{ex:ex-defs-1}
Let $\A = \{1, 2, 3, 4, 5\}$ be a set of arguments, and let $F = \langle \A, \R \rangle$ 
be the argumentation framework represented by the attack graph of Figure 
\ref{fig:ex-defs-1}.
We will use this argumentation framework as a running example all along the paper.
It has no self-attacking arguments.
The set of arguments $345$ is self-defending but it is not conflict-free.
In fact, $45$ is self-defending, and a preferred extension of $\AF$.
\begin{figure}[ht!]
\centering 
\includegraphics[scale=1.1, page=1]{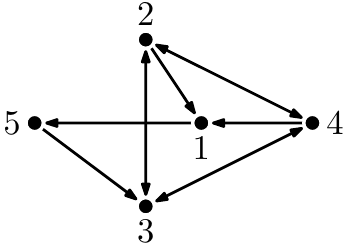}%
\caption{The attack graph $G = (\A, \R)$ of $F$.}
\label{fig:ex-defs-1}
\end{figure}
\end{example}

\paragraph{Closure systems and implicational systems.}
We use terminology from \cite{gratzer2011lattice, wild2017joy}.
Let $\F$ be a set system over $\A$.
The pair $\langle \F, \subseteq \rangle$ denotes the family $\F$ ordered by 
set-inclusion.
The system $\F$ is a \emph{closure system} over $\A$ if $\A \in \F$ and $F_1 \cap F_2 \in 
\F$ for every $F_1, F_2 \in \F$, that is, if $\F$ is closed under intersection.
The sets in $\F$ are called \emph{closed sets}.
It is known that if $\F$ is a closure system,
$\langle \F, \subseteq \rangle$ is a \emph{lattice} (and so does $\langle \bar{\F}, 
\subseteq \rangle$).

An \emph{implication} over $\A$ is an expression of the form $X \rightarrow y$ where 
$X \subseteq \A$ and $y \in \A$.
In $X \rightarrow y$, $X$ is the \emph{premise} and $y$ the \emph{conclusion} of the 
implication.
An \emph{implicational system} $\IS$ over $\A$ is a set of implications over $\A$.
Let $\IS = \{X_1\rightarrow x_1,\dots,X_n\rightarrow x_n\}$ be an implicational system 
on a set $\A$ and let $S \subseteq \A$. 
The \emph{$\IS$-closure} of $S$, denoted  $S^\IS$, is the inclusion-wise minimal set 
containing $S$ and satisfying:
\[ X_j \subseteq S^\Sigma \implies x_j \in S^\Sigma \quad \text{ for every } 1 \leq j 
\leq n. \]
The family  $\F_\IS = \{S^\IS \mid S \subseteq \A\}$ is a closure system. 

\begin{example} \label{ex:ex-defs-2}
Let $\A = \{1, 2, 3, 4, 5\}$ and consider the implicational system $\IS = \{1 \imp 3, 24 
\imp 5, 34 \imp 1, 23 \imp 1\}$.
For example, $124$ is not $\IS$-closed as it contains the premise of $1 \imp 3$, but not 
its conclusion.
\end{example}

It has been shown in \cite{elaroussi:hal-03184819} that the
self-defending sets of an argumentation framework coincide with the closed sets of an implicational system. 

\begin{theorem}\cite{elaroussi:hal-03184819}
	\label{thm:thmb}
	Let $\AF = \langle \A,\R \rangle$ be an argumentation framework, and let  	
	$\IS=\{y^-\rightarrow z \mid (y,z)\in \R \text{ and } (z,y)\notin \R\}$ be its 
	associated implicational system. Then $\bar{\F_{\IS}} = \SD(\AF)$.	
\end{theorem}

\begin{example}[Running example]
The implicational system associated to $F$ is $\IS = \{1 \imp 3, 24 \imp 5, 34 \imp 1, 23 
\imp 1\}$ (the implicational system of Example \ref{ex:ex-defs-2}).
For example, the implication $34 \imp 1$ is derived from the fact that $2$ attacks $1$ 
and $2^- = \{3, 4\}$.
In Figure \ref{fig:ex-defs-3}, we represent the lattice $\langle \SD(F), \subseteq 
\rangle$.
Due to Theorem \ref{thm:thmb}, we have $\SD(F) = \bar{\F_{\IS}}$ where $\F_{\IS}$ is the 
closure system of $\IS$ in Example \ref{ex:ex-defs-2}.
\begin{figure}[ht!]
\centering 
\includegraphics[scale=1, page=2]{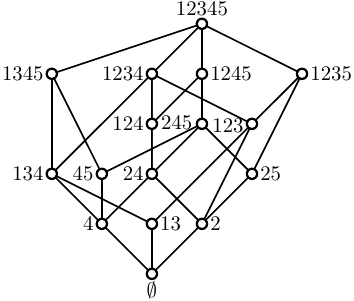}%
\caption{The lattice $\langle \SD(F), \subseteq \rangle$ of self-defending sets of $F$}
\label{fig:ex-defs-3}
\end{figure}

\end{example}

\paragraph{Enumeration complexity.} We use the definitions of 
\cite{johnson1988generating}.
Let $\ctt{A}$ be an algorithm with input $x$ and output a set of solutions $R(x)$.
We denote by $\card{R(x)}$ the number of solutions in $R(x)$.
We assume that each solution in $R(x)$ has size $\poly(\card{x})$.
The algorithm $\ctt{A}$ is running in \emph{output-polynomial} time if its execution time 
is bounded by $\poly(\card{x} + \card{R(x)})$.
If the delay between two solutions output and after the last one is $\poly(\card{x})$, 
$\ctt{A}$ has \emph{polynomial-delay}.

\section{Naive-bijective argumentation frameworks}
\label{sec:naive-bijective }

In this section, we show that checking whether the family of preferred extensions of a 
given 
argumentation framework is equal to the family of its naive sets is a 
\csf{coNP}-complete problem. 
In contrast, we give a polynomial time algorithm to analyze the bounded in-degree 
argumentation frameworks.

\begin{definition} \label{def:naive-bijective }
Let $\AF=\langle \A,\R \rangle$ be an argumentation framework.
We say that $\AF$ is \emph{naive-bijective } if $\NAIV(\AF) = \PREF(\AF)$.
\end{definition}
In order to illustrate the above definition, we provide the following examples.

\begin{example}[Running Example]
The argumentation framework $\AF$ of Example \ref{ex:ex-defs-1} is naive-bijective.
Indeed, we have $\NAIV(\AF) = \PREF(\AF) = \{13, 45, 25\}$.
\end{example}

\begin{example}
	\label{ex: non-bijective}
Let $\A = \{1, 2, 3\}$ and $\R = \{(1, 2), (2, 3)\}$. 
The set $\{2\}$ is naive, but not self-defending. 
Thus, the argumentation framework $\langle \A, \R \rangle$ is not naive-bijective.
\end{example}

Since naive sets can be enumerated with polynomial delay and space \cite{johnson1988generating}, 
naive-bijective argumentation frameworks are particularly interesting for the enumeration 
of preferred extensions.

The following lemma provides a simple case where an argumentation framework fails to be 
naive-bijective.
\begin{lemma} \label{lem:lemSym1}
Let $\AF=\langle \A,\R\rangle$ be an argumentation framework and let $\IS$ be its 
associated implicational system, obtained with Theorem \ref{thm:thmb}. 
If $\emptyset^{\IS} \nsubseteq \S(\AF)$, then $\AF$ is not naive-bijective.
\end{lemma}

\begin{proof}
The result follows from the fact that an argument $x \in \emptyset^{\IS}\setminus 
\S(\AF)$ belongs to at least one conflict-free set in $\AF$ and not to any admissible 
set in $\AF$.
\end{proof}
An argument $x$ in $\emptyset^{\IS}$ cannot appear in a self-defending set.
In particular, it is not contained in any admissible set.
Thus, adding the self-attack $(x, x)$ to $\R$ does not change $\ADM(\AF)$.
Therefore, an argumentation framework where $\emptyset^\IS \nsubseteq \S(F)$ can 
be 
transformed in polynomial time to an argumentation framework satisfying $\emptyset^{\IS} 
\subseteq \S(F)$ without changing the admissible and preferred extensions.
This is done by adding a self-attack to each argument $x$ in 
$\emptyset^\IS$.
\begin{example}
Consider the argumentation framework of Example \ref{ex: non-bijective}. 
$\AF$ can be transformed into a naive-bijective argumentation framework. 
Indeed, $2\in \emptyset^\IS$ as $1^- = \emptyset$, and $\{2\}$ is the unique naive set 
which is not a preferred extension of $\AF$.
Thus, adding the loop $(2, 2)$ to $\R$ does not change $\PREF(\AF)$, while it removes 
$\{2\}$ from $\NAIV(\AF)$.
As a result, $\AF$ becomes naive-bijective. 
\end{example}

The next theorem characterizes naive-bijective  argumentation frameworks.

\begin{theorem}\label{thm:ThmSym2}
Let $\AF=\langle \A,\R \rangle$ be an argumentation framework and let $\IS$ be its 
associated implicational system, obtained with Theorem \ref{thm:thmb}. 
Suppose that  $\emptyset^\IS\subseteq \S(\AF)$.
Then, $\AF$ is not naive-bijective  if and only if for some $x, y \in \A$ such 
that $(y, x) \in \R$ and $(x, y) \notin \R$, there exists a conflict-free set $S 
\subseteq \Gamma(y^-) \setminus y^-$ satisfying:
\begin{itemize}
\item $S \cup \{x\} \in \CF(\AF)$,
\item $S \cap \Gamma(z) \neq \emptyset$ for each $z \in y^-$.
\end{itemize}

\end{theorem}

\begin{proof}
We begin with the if part.
Let $x,y\in \A$ such that $(y,x)\in \R$ and $(x,y)\notin \R$ and 
suppose that there exists $S\subseteq \Gamma(y^-) \setminus y^-$ such that 
$S\cup \{x\} \in \CF(\AF)$ and for all $z \in y^-$, $S \cap \Gamma(z) \neq 
\emptyset$. 
Then, there exists a naive set $S'$ such that $S \cup \{x\} \subseteq S'$ and 
$S' \cap y^- = \emptyset$, i.e. such that $y$ attacks $x$ and $y$ is not attacked by 
$S'$.
Thus, $S'$ is not self-defending and hence it is not preferred either.

We move to the only if part.
Assume that $\AF$ is not naive-bijective.
Then, there exists a naive set $S'$ which is not self-defending.
More precisely, there exists $x \in S'$ and $y \notin S'$ such that $(y, x) \in \R$ and 
$S' \cap y^- = \emptyset$.
In particular, we have $(x, y) \notin \R$.
Now, let $S = S' \cap \Gamma(y^-)$.
Since $S' \cap y^- = \emptyset$, we have that $S \subseteq \Gamma(y^-) \setminus y^-$.
We show that $S$ satisfies the two conditions of the theorem:
\begin{itemize} \itemsep0em 
\item $S \cup \{x\} \in \CF(\AF)$.
It follows from $S \cup \{x\} \subseteq S'$ and the fact that $S'$ is conflict-free.

\item $S \cap \Gamma(z) \neq \emptyset$ for each $z \in y^-$.
Since $S'$ is a naive set and $S' \cap y^- = \emptyset$, for every $z \in y^-$ we have 
that $S' \cup \{z\} \notin \CF(F)$.
Thus, $S'$ contains some $t \in \Gamma(z)$.
As $\Gamma(z) \subseteq \Gamma(y^-)$, we deduce that $t \in S' \cap \Gamma(y^-) = S$ 
and hence that $S \cap \Gamma(z) \neq \emptyset$. 
\end{itemize}
In other words, there exists $x, y \in \A$ with $(y, x) \in \R$, $(x, y) \notin \R$ and 
such that some $S \subseteq \Gamma(y^-) \setminus y^-$ satisfies $S \cup \{x\} \in 
\CF(\AF)$ and $S \cap \Gamma(z) \neq \emptyset$ for each $z \in y^-$.
This concludes the proof of the theorem.

\end{proof}

Now we consider the problem of recognizing naive-bijective argumentation 
frameworks.
Formally, the problem reads as follows:

\begin{decproblem}
	\problemtitle{Naive-bijective  argumentation framework recognition (\csmc{NAF})}
	\probleminput{An argumentation framework $\AF = \langle \A, \R \rangle$.}
	\problemquestion{\textit{Is $\AF$ naive-bijective ?}} 
\end{decproblem}

On the basis of Theorem \ref{thm:ThmSym2}, we deduce the following strategy to solve 
\csmc{NAF}.
For each $x,y\in \A$ such that $(y,x)\in \R$ and $(x,y)\notin \R$, check 
whether there exists a 
conflict-free set $S \subseteq \Gamma(y^-) \setminus y^-$ such that $S\cup \{x\} 
$ is a conflict-free set of $\AF$ and for all $z\in y^-$, $S \cap \Gamma(z) \neq 
\emptyset$. 
Return \textit{``No''} if such a $S$ exists and \textit{``Yes''} otherwise.
Quite clearly, this algorithm has exponential-time complexity.
Yet, it runs in polynomial time if the in-degree of $\AF$ is constant.

\begin{proposition} \label{prop:AlgSym}
	Let $\AF=\langle \A,\R \rangle$ be an argumentation framework.
	If the in-degree of $\AF$ is bounded by a constant, there exists a polynomial time 
	algorithm (in the size of $\AF$) to check whether $\AF$ is naive-bijective .
\end{proposition}

\begin{proof}
Let $x, y \in \A$ such that $(y, x) \in \R$ and $(x, y) \notin \R$.
By assumption, $y^- = \{z_1, \dots, z_k\}$ for some constant $k$.
For convenience, let $\cc{T} = \{S \subseteq \Gamma(y^-) \setminus y^- \mid S \cap 
\Gamma(z_i) \neq \emptyset \text{ for each } z_i \in y^-\}$.
According to Theorem \ref{thm:ThmSym2}, $\AF$ is not naive-bijective  if there 
exists 
some conflict-free set $S \in \cc{T}$ such that $S \cup \{x\} \in \CF(\AF)$.
Since every subset of a conflict-free set is conflict-free, it is sufficient to check the 
existence of such a $S$ in $\min_{\subseteq}(\cc{T})$, the collection of inclusion-wise 
minimal elements of $\cc{T}$.
Let $S \in \min_{\subseteq}(\cc{T})$.
Since $\card{y^-} = k$ and by definition of $\cc{T}$, $S$ has size at most $k$.
Hence, $\min_{\subseteq}(\cc{T})$ can be computed in polynomial time by checking all the 
subsets of $\Gamma(y^-) \setminus y^-$ of size at most $k$.
Then, for each $S \in \min_{\subseteq}(\cc{T})$, testing whether $S \cup \{x\} 
\in \CF(\AF)$ also requires polynomial time in the size of $\AF$.

Applying this algorithm for each $x, y \in \A$ such that $(y, 
x) \in \R$ and $(x, y) \notin \R$, we can verify the condition of Theorem 
\ref{thm:ThmSym2} in polynomial time in the size of $\AF$, concluding the proof.
\end{proof}

We settle the general hardness of \csmc{NAF} in the subsequent theorem.

\begin{theorem} \label{thm:naf-npc}
	The problem \csmc{NAF} is \csf{coNP}-complete.
\end{theorem}

\begin{proof}
The proof is a reduction from 3-SAT.
Let $C = \bigwedge_{i=1}^m C_i$ be a CNF formula, where
each clause $C_i$ is a disjunction of three literals from $\{x_1,...,x_n, \neg 
x_1,...,\neg x_n\}$.
Let $\AF$ be the argumentation framework $\langle \A, \R \rangle$ where (see also 
Figure \ref{fig:attack-graph-1}):
\begin{align*}
\A = & \{x, y\} \cup\{C_i \mid 1\leq i\leq m\} \\
     & \cup \{x_j, \neg x_j \mid 1\leq j\leq n\}; \\ 
\R = & \{(y,x)\} \cup \{(C_i, y),(y, C_i) \mid 1\leq i\leq m\} \\
     & \cup \{(x_j, C_i),(C_i, x_j) \mid x_j \text{ occurs in } C_i\} \\
     & \cup \{(\neg x_j,C_i),(C_i, \neg x_j) \mid \neg x_j \text{ occurs in } C_i\} \\
     & \cup \{(x_j,\neg x_j), (\neg x_j, x_j) \mid 1\leq j\leq n\}.
\end{align*}
The construction of $\AF$ is done in polynomial time in the size of $C$.
We show that $\AF$ has a non-admissible naive set if and only if $C$ is 
satisfiable.
Let $S$ be a naive set of $\AF$.
We have two cases.
First, $y \in S$.
Then, $S$ is a conflict-free extension of the sub-argumentation framework $\AF_{\vert 
x}= 
\langle \A \setminus \{x\}, \R\setminus \{(y,x)\} \rangle$.
Since $\AF_{\vert x}$ is symmetric and $x^+ = \emptyset$, each subset of $\A 
\setminus 
\{x\}$ is self-defending in $\AF$ (see \cite{coste2005symmetric}).
In particular, every conflict-free set of $\AF_{\vert x}$ is admissible in $\AF$. 
Therefore, $S$ must be admissible in $\AF$.
The second case is, $y \notin S$.
Then, $x \in S$ and $S$ is not admissible in $\AF$ if and only if $S \cap \{C_1, 
\dots, C_m\} = \emptyset$.
By construction of $\AF$, it follows that $S$ is not admissible if and only if for 
each 
$C_i$, $1 \leq i \leq m$, there exists some $1 \leq j \leq n$ such that $\ell_j \in 
C_i$ 
and $\ell_j \in S$ with $\ell_j \in \{x_j, \lnot x_j\}$.
As $(x_j, \lnot x_j) \in \R$ for each $1 \leq j \leq n$ and $S$ is conflict-free, 
$\{x_j, 
\lnot x_j\} \subseteq S$ cannot hold.
We deduce that $S$ is not admissible if and only if $C$ is satisfiable, concluding 
the 
proof.
\end{proof}
\begin{figure}[h!]
	\centering 
	\includegraphics[scale=1, page=1]{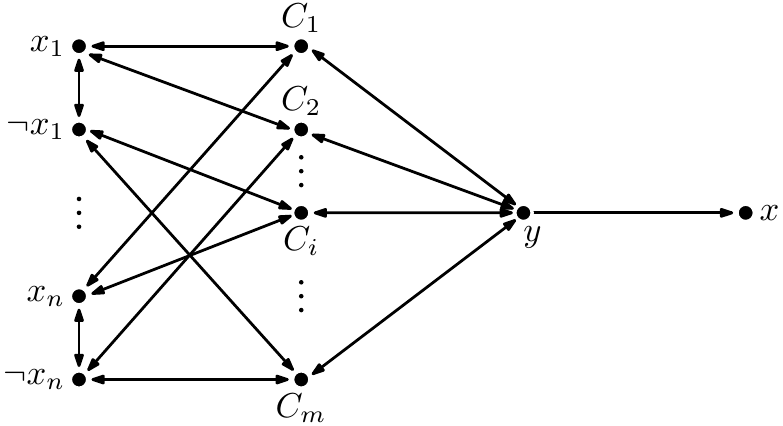}%
	\caption{Attack graph of $\AF$.}
	\label{fig:attack-graph-1}
	\label{FigSym}
\end{figure}
\section{Naive-recasting and admissible-closed argumentation frameworks}
\label{sec:admissible-closed}

In this section, we study the class of argumentation frameworks whose preferred 
extensions are the naive sets of another argumentation framework.
This class of argumentation frameworks is called naive-recasting and defined formally as 
follows.

\begin{definition} \label{def:naive-recasting}
	Let $\AF=\langle \A,\R \rangle$ be an argumentation framework. We say that $F$ is 
	\emph{naive-recasting} if there exists an argumentation framework $\AF'=\langle 
	\A',\R' 
	\rangle$ such that $\PREF(\AF)=\NAIV(\AF')$.
\end{definition} 

Dunne \textit{et al.} \cite{dunne2015characteristics} show that it is \csf{coNP}-hard to 
recognize 
naive-recasting argumentation frameworks. 
This problem is defined as follows:

\begin{decproblem}
\problemtitle{Naive-recasting argumentation framework recognition (\csmc{NRAF})}
\probleminput{An argumentation framework $\AF = \langle \A, \R \rangle$.}
\problemquestion{\textit{Is $\AF$ naive-recasting?}} 
\end{decproblem} 


Let $\AF=\langle \A,\R\rangle$ be an argumentation framework.
We prove that $\AF$ is naive-recasting in the restricted case where $\ADM(\AF)$ is 
closed under intersection. 
First, we formally define this class of argumentation frameworks.

\begin{definition} \label{def:admissible-closed}
Let $\AF=\langle \A,\R\rangle$ be an argumentation framework.
We say that $\AF$ is \emph{admissible-closed} if its admissible sets are closed 
under 
intersection.
\end{definition}
Let us now provide some examples to illustrate the above definition.

\begin{example}[Running Example]
The argumentation framework $\AF$ of Example \ref{ex:ex-defs-1} is not admissible-closed.
Indeed, $45$ and $25$ are admissible but $45 \cap 25 = 5$ is not self-defending. 
\end{example}

\begin{example}
Let $\A = \{1, 2, 3, 4\}$ and $\R = \{(1, 3), (3, 1), (1, 2), (3, 2), (3, 4)\}$.
Let $\AF = \langle \A, \R \rangle$.
We have $\ADM(\AF) = \{\emptyset, 1, 3, 14\}$.
Since $1 \cap 3 = 14 \cap 3 = \emptyset$ and $1 \cap 14 = 1$, we derive that $\AF$ is 
admissible-closed.
\end{example}

In the following lemma, we show that each admissible-closed argumentation framework has a 
family  of self-defending sets closed under union and intersection and includes all its 
admissible sets.

\begin{lemma} \label{lem:ThmDistributive}
	Let $\AF=\langle \A,\R\rangle$ be an admissible-closed argumentation framework. Then, 
	there exists a family of self-defending sets $SD'(\AF)\subseteq SD(\AF)$ closed under 
	union and intersection such that $SD'(\AF)\cap CF(\AF)$= $ADM(\AF)$.
\end{lemma}

\begin{proof}
	Let $\SD'(\AF)$ be the set system obtained from $\ADM(\AF)$ after closing under 
	union, 
	that is $\SD'(\AF) = \{ \bigcup \I \mid \I \subseteq \ADM(\AF) \}$.
	Since self-defending sets are closed under union \cite{dung1995acceptability}
	and 
	$\ADM(\AF) \subseteq \SD(\AF)$, $\SD'(\AF) \subseteq \SD(\AF)$ readily holds.
	Then, the fact that $\SD'(\AF)\cap \CF(\AF) = \ADM(\AF)$ follows from $\ADM(\AF) 
	\subseteq 
	\SD'(\AF) \subseteq \SD(\AF)$ and $\SD(\AF)\cap \CF(\AF) = \ADM(\AF)$.
	
	It remains to prove that $\SD'(\AF)$ is closed under intersection.
	Let $S_1, S_2 \in \SD'(\AF)$. 
	We show that $S_1 \cap S_2 \in \SD'(\AF)$.
	By construction of $\SD'(\AF)$, $S_1 = \bigcup_{i=1}^k A_i$ and $S_2 = 
	\bigcup_{j=1}^{\ell} B_j$ with $A_i, B_j \in \ADM(\AF)$ for every $1 \leq i \leq k$ 
	and 
	every $1 \leq j \leq \ell$.
	Therefore, we have:
	\[ 
	S_1 \cap S_2 = \left(\bigcup_{i=1}^k A_i\right) \cap \left(\bigcup_{j=1}^{\ell} 
	B_j\right) = \bigcup_{i=1}^k \bigcup_{j=1}^{\ell} (A_i\cap B_j)
	\]
	However, $\ADM(\AF)$ is closed under intersection by assumption.
	Consequently, $A_i \cap B_j \in \ADM(\AF)$ for every $1 \leq i \leq k$ and every $1 
	\leq 
	j \leq \ell$.
	We deduce that $S_1 \cap S_2 \in \SD'(\AF)$ as required.
\end{proof}
We are now ready to state the main result of this section.
\begin{theorem}	\label{tmh:ThmDsitrEx}
	Every admissible-closed argumentation framework is naive-recasting.
\end{theorem}

\begin{proof}
	Consider an admissible-closed argumentation framework $\AF=\langle \A, \R \rangle$.
	According to Lemma \ref{lem:ThmDistributive}, there is a collection  of 
	self-defending 
	sets $\SD'(\AF) \subseteq \SD(\AF)$ which completely 
	includes $\ADM(\AF)$. 
	Since $\SD'(\AF)$ is closed under union and intersection, it is possible to 
	represent the closure system $\langle \SD'(\AF), \subseteq \rangle$ by an
	implicational system $\IS_d$ with singleton premises \cite{wild2017joy}.
	Let $S^*$ be the inclusion-wise maximal set of $\SD'(\AF)$. 
	Consider now the argumentation framework $\AF'=\langle \A,\R'\rangle$ constructed 
	from 
	$\AF$, $\IS_d$ and $S^*$, by extending $\R$ to $\R'$ with the following rules:
	\begin{itemize}
		\item [ ] \textbf{Rule 1:} if $x\in \A\setminus S^*$, then add $(x,x)$ to $\R'$.
		\item [ ] \textbf{Rule 2:} if $x\rightarrow z \in \IS_d$ and $y\in \Gamma(z)$, 
		then 
		add $(x,y)$ and $(y,x)$ to $\R'$. 
	\end{itemize}
	To show that $F$ is naive-recasting, it is sufficient to prove 
	$\NAIV(\AF') \subseteq \ADM(\AF) \subseteq \CF(\AF')$.
	
	We begin with the first inclusion.
	Let $S \in \NAIV(\AF')$.
	Since $\R \subseteq \R'$, $S \in \CF(\AF)$ is clear.
	By \textbf{Rule 1}, $S \subseteq S^*$.
	We show that $S$ is closed for $\IS_d$ and hence that $S$ is self-defending in $\AF$, 
	in 
	virtue of Lemma \ref{lem:ThmDistributive}.
	Let $x \rightarrow z$ be an implication of $\IS_d$ such that $x \in S$.
	Since $S$ is conflict-free in $\AF'$, \textbf{Rule 2} implies that $S \cap \Gamma(z) 
	= 
	\emptyset$.
	Thus, $S \cup \{z\}$ is also conflict-free in $\AF'$.
	As $S$ is an inclusion-wise maximal conflict-free set of $\AF'$, we deduce that $z 
	\in S$.
	Therefore, $S$ is a closed set of $\IS_d$ so that $S \in \SD(\AF)$ holds.
	Consequently, we have $S \in \ADM(\AF)$ as required.
	
	We move to the second inclusion.
	Let $S \in \ADM(\AF)$.
	We show that $S \in \CF(\AF')$.
	By assumption on $\AF$, $S \subseteq S^*$ and $S$ is a closed set  of $\IS_d$.
	Suppose for contradiction that $S$ is not conflict-free in $\AF'$.
	Since $S$ is conflict-free in $\AF$ and $S \subseteq S^*$, it must be that $S$ 
	contains a 
	pair $(x, y)$ of $\R'$ which has been added by \textbf{Rule 2}.
	Consequently, there exists an implication $x\rightarrow z$ in $\IS_d$ such that $x 
	\in S$ 
	and $y \in \Gamma(z)$.
	However, $S$ is a closed set  of $\IS_d$ by assumption on $\AF$ and by construction 
	of $\IS_d$.
	Therefore, $x \in S$ implies $z \in S$ so that $\{y, z\} \subseteq S$ also holds.
	As $y \in \Gamma(z)$ in $\AF$, this contradicts $S \in \CF(\AF)$.
	We deduce that $S \in \CF(\AF')$, which concludes the proof.
\end{proof}

In the following result, we show that it is \csf{coNP}-complete to recognize 
admissible-closed argumentation frameworks.

\begin{decproblem}
	\problemtitle{Admissible-closed argumentation framework recognition (\csmc{ACAF})}
	\probleminput{An argumentation framework $\AF = \langle \A, \R \rangle$.}
	\problemquestion{\textit{Is $\AF$ admissible-closed?}} 
\end{decproblem} 

\begin{theorem} \label{thm:acaf-npc}
	The problem \csmc{ACAF} is \csf{coNP}-complete.
\end{theorem}

\begin{proof}
The proof is a reduction from 3-SAT. 
Let $C = \bigwedge_{i=1}^m C_i$ be a CNF formula, where
each clause $C_i$ is a disjunction of three literals from $\{x_1,...,x_n, \neg 
x_1,...,\neg x_n\}$.
Let $\AF$ be the argumentation framework $\langle \A, \R \rangle$ where (see also 
Figure 
\ref{fig:attack-graph-2}):
\begin{align*}
\A = & \{x, y, z_1, z_2\} \cup \{C_i \mid 1\leq i\leq m\} \\
     & \cup \{x_j, \neg x_j \mid 1 \leq j\leq n\} \\
\R = & \{(x_j,C_i),(C_i,x_j) \mid x_j \text{ occurs in } C_i\} \\
     & \cup \{(\neg x_j,C_i),(C_i, \neg x_j) \mid \neg x_j \text{ occurs in } C_i\} \\ 
     & \cup \{(z_1,y), (z_2,y), (z_1,z_2), (z_2,z_1), (y,x)\} \\
     & \cup \{(C_i,z_1) \mid 1\leq i\leq m\} \\
     & \cup \{(x_j,\neg x_j), (\neg x_j, x_j) \mid 1\leq j\leq n\}
\end{align*}
The construction of $\AF$ can be done in polynomial time in the size of $C$.
Using the same reasoning as in Theorem \ref{thm:naf-npc}, we can show that $\AF$ 
has an admissible set $S$ with $\{x, z_1\} \subseteq S$ if and only if $C$ is 
satisfiable.
It remains to show that $\AF$ is admissible-closed if and only if such a $S$ does not 
exist.

We prove the only if part using contrapositive.
Assume that there exists an admissible set $S$ such that $\{z_1, x\} \subseteq 
S$.
Since $(z_1, z_2) \in \R$ and $S$ is admissible, $z_2 \notin S$.
As $\{x, z_2\}$ is also admissible, we have that $S \cap \{x, z_2\}$ is not 
self-defending as $x \in S \cap \{x, z_2\}$ and $y^- \cap S = \emptyset$.
Hence, $\AF$ is not admissible-closed.

We move to the if part.
Suppose that $S$ does not exists. 
By assumption, an admissible set which contains $x$ also contains $z_2$.
Moreover, $y^- = \{z_1, z_2\}$ and $(z_1, z_2) \in \R$ so that no admissible set 
in 
$\AF$ contains $y$.
Thus, the family of admissible sets of $\AF$ is not closed under intersection if 
and 
only 
if there exists two admissible sets $S_1, S_2 \subseteq \B= \A\setminus 
\{x,y,z_1,z_2\}$ 
such 
that their intersection $S_1 \cap S_2$ is not an admissible set of $\AF$.
However, there is no attack from $\{x, y, z_1, z_2\}$ to $\B$ and the 
sub-argumentation 
framework $\AF' = \langle \B, \R\cap (\B\times \B)\rangle$ is symmetric (see, 
\cite{coste2005symmetric}).
Therefore, each conflict-free set of $\AF'$ is admissible in $\AF$.
As the intersection of two conflict-free sets is conflict-free, we deduce that for 
every 
admissible sets $S_1, S_2 \subseteq \B$, $S_1 \cap S_2$ is also admissible.
As a consequence, $\AF$ is admissible-closed, concluding the proof.
\end{proof}

\begin{figure}[h!]
	\centering 
	\includegraphics[scale=1, page=2]{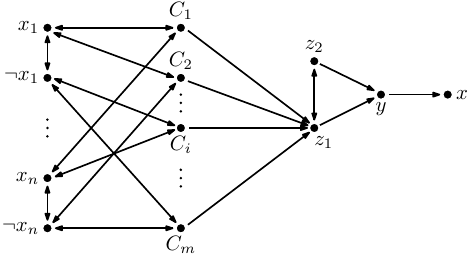}%
	\caption{Attack graph of $\AF$.}
	\label{fig:attack-graph-2}
\end{figure}

To conclude this section, let $\AF$ be an admissible-closed argumentation framework.
According to Theorem \ref{tmh:ThmDsitrEx}, we can build an argumentation framework $\AF'$ 
(on the same set of arguments) such that $\NAIV(\AF') = \PREF(\AF)$.
Thus, we can compute $\PREF(\AF)$ with polynomial delay from $\AF'$ using the algorithm 
of Johnson \textit{et al.} \cite{johnson1988generating}.
However, whether we can find $\AF'$ in polynomial time in the size of $\AF$ is open.

\paragraph{Open problem.} Let $\AF$ be an admissible-closed argumentation framework and 
let $\AF'$ be an argumentation framework such that $\NAIV(\AF') = \PREF(\AF)$.
Is it possible to find $\AF'$ in polynomial time from $\AF$?

\section{Irreducible self-defending sets and preferred extensions}
\label{sec:irreducible}

In this section, we introduce the irreducible self-defending sets of an argumentation 
framework.
They are those self-defending sets that cannot be obtained as the union of other 
self-defending sets.
In lattice theoretic terms \cite{gratzer2011lattice}, they are the 
join-irreducible elements of the lattice of self-defending sets.
Using these irreducible elements, we construct a new argumentation framework whose naive 
sets are in a one-to-one correspondence with the preferred extensions of the input 
framework.
We deduce an algorithm to list the preferred extensions of an argumentation framework 
using its irreducible self-defending sets.
This algorithm has polynomial-delay and polynomial space when the irreducible 
sets are given.
First, we formally define irreducible self-defending sets.

\begin{definition} \label{def:irreducible}
	Let $\AF = \langle \A, \R \rangle$ be an argumentation framework.
	A self-defending set $S \in \SD(\AF)$ with $S \neq \emptyset$ is \emph{irreducible} 
	if 
	for every $S_1, S_2 \in \SD(\AF)$, $S = S_1 \cup S_2$ implies $S = S_1$ or $S = S_2$.
	We denote by $\IRR(\AF)$ the family of irreducible self-defending sets of $\AF$.
\end{definition}

\begin{remark}
	The \emph{initial admissible sets} \cite{xu2018initial} are irreducible 
	self-defending 
	sets.
	A set of arguments is an \emph{initial admissible set} if it is an 
	inclusion-wise 
	minimal 
	non-empty admissible set.
\end{remark}

Let $\AF = \langle \A, \R \rangle$ be an argumentation framework.
The family $\SD(\AF)$ can be generated from $\IRR(\AF)$ by taking the union 
of all possible subsets of $\IRR(\AF)$, that is $\SD(\AF) = \{\bigcup \I \mid \I 
\subseteq \IRR(\AF)\}$, see \cite{gratzer2011lattice}.
Since a self-defending set in $\IRR(\AF)$ cannot be obtained by union, 
$\IRR(\AF)$ is in fact the most compact representation of $\SD(\AF)$ by self-defending 
sets.
For a given $S \in \SD(\AF)$, we put $\IRR(S) = \{S' \in \IRR(\AF) \mid  S' \subseteq 
S\}$.
We have $S = \bigcup \IRR(S)$.

\begin{remark}
	This representation by irreducible elements is common in 
	several fields of computer science where lattices are used (see e.g.
	\cite{doignon2012knowledge, ganter2010two, markowsky1975factorization, wild2017joy}).
\end{remark}

\begin{example}[Running Example]
In the argumentation framework $\AF$ of Example \ref{ex:ex-defs-1} (see also Figure 
\ref{fig:ex-defs-3}), we have $\IRR(\AF) = \{4, 45, 2, 25, 13, 124\}$.
For example, $245$ is a self-defending set which is not irreducible.
It is obtained as the union of $25$ and $24$.
\end{example}

Let $\AF = \langle A, \R \rangle$ be an argumentation framework.
We construct an argumentation framework $\AF_{\IRR} = \langle \IRR(\AF), \R_{\IRR} 
\rangle$ from which $\PREF(\AF)$ can be easily recovered with $\NAIV(\AF_{\IRR})$.
We start with the following lemma.

\begin{lemma} \label{lem:median}
	Let $\AF = \langle \A, \R \rangle$  be an argumentation framework.
	A self-defending set $S$ is not admissible if and only if there exists $S_1, S_2 \in 
	\IRR(S)$ such that $S_1 \cup S_2$ is not an admissible set of $\AF$.
\end{lemma}

\begin{proof}
	The if part is clear.
	Let $S \in \SD(\AF)$ with $S \notin \ADM(\AF)$.
	By definition, $S$ contains some pair $(x, y)$ such that $(x, y) \in \R$.
	Since $S = \bigcup \IRR(S)$, there exist $S_1, 
	S_2 \in \IRR(S)$ such that $x \in S_1$ and $y \in S_2$ (possibly $S_1 = 
	S_2$).
	Thus, $S_1 \cup S_2$ is a self-defending set which is not admissible, 
	concluding the proof.
\end{proof}

In the next lemma, we use Lemma \ref{lem:median} to devise an argumentation 
framework $\AF_{\IRR}$ on $\IRR(\AF)$ and show the relationship between 
$\NAIV(\AF_{\IRR})$ and $\PREF(\AF)$.

\begin{lemma} \label{lem:bijection}
Let $\AF = \langle \A, \R \rangle$ be an argumentation framework.
Let $\AF_{\IRR} = \langle \IRR(\AF), \R_{\IRR} \rangle$ be a new argumentation 
framework 
where $\R_{\IRR} = \{(S_1, S_2) \mid S_1, S_2 \in \IRR(\AF) \text{ and } S_1 \cup S_2 
\notin \ADM(\AF) \}$.
Then, the following equality holds and it is moreover a bijection between 
$\PREF(\AF)$ 
and $\NAIV(\AF_{\IRR})$:
\[
\PREF(\AF) = \left\{ \bigcup \I \mid \I \in \NAIV(\AF_{\IRR}) \right\}
\]
\end{lemma}

\begin{proof}
First, we show the equality.
Let $\AF_{\IRR}$ be the argumentation framework defined in the lemma.
We begin with the $\subseteq$ part of the equality.
Let $S \in \PREF(\AF)$ and consider $\IRR(S)$.
We show that $\IRR(S)$ is a naive set of $\AF_{\IRR}$.
Since $S$ is admissible, $\IRR(S)$ is a conflict-free extension of $\AF_{\IRR}$ by 
Lemma \ref{lem:median} and by construction of $\AF_{\IRR}$.
If $\IRR(S) = \IRR(\AF)$, then it is clearly the unique naive set of 
$\AF_{\IRR}$.
Suppose now there exists $S' \in \IRR(\AF)$ such that $S' \notin \IRR(S)$ and 
consider 
the set $S'' = S' \cup \bigcup \IRR(S) = S' \cup S$.
Since the union of two self-defending set is self-defending, $S''$ is a 
self-defending 
set of $\AF$.
Moreover, we have $S \subset S''$ by definition of $\IRR(S)$.
Since $S \in \PREF(\AF)$ and $S \subset S''$, we have $S'' \notin \ADM(\AF)$.
Thus, there exists a pair $(x, y)$ in $\R$ such that $\{x, y\} \subseteq S''$ while 
$\{x, 
y\} \nsubseteq S$.
As $S$ is admissible and $S'' = S \cup S'$ we deduce that either $x \in 
S'$ or $y \in S'$.
If both $x$ and $y$ belong to $S'$, then $\IRR(S) \cup \{S'\}$ is not a 
conflict-free extension of $\AF_{\IRR}$.
Assume now that $x \in S'$ but $y \notin S'$, without loss of generality.
Observe that we must have $y \in S$ as otherwise $\{x, y\} \subseteq S''$ would not 
hold.
Because $S = \bigcup \IRR(S)$ and $y \in S$, we deduce that there exists 
$T \in \IRR(S)$ such that $y \in T$.
Hence, $\{x, y\} \subseteq S' \cup T$ which entails that $(T, S') \in 
\R_{\IRR}$ and $\IRR(S) \cup \{S'\}$ is not a conflict-free extension of $\AF_{\IRR}$.
Consequently, for every $S' \in \IRR(\AF) \setminus \IRR(S)$, we have that $\IRR(S) 
\cup 
\{S'\}$ is not a conflict-free extension of $\AF_{\IRR}$.
Together with the fact that $\IRR(S)$ is conflict-free, we conclude that $\IRR(S)$ is 
a naive set of $\AF_{\IRR}$ as expected.
	
We move to the $\supseteq$ part.
Let $\I \in \NAIV(\AF_{\IRR})$. 
Let $S = \bigcup \I$.
We show that $S \in \PREF(\AF)$.
Since $\IRR(\AF)$ generates $\SD(\AF)$, $S \in \SD(\AF)$ readily holds.
Assume for contradiction that $S \notin \ADM(\AF)$.
Then, there exists $(x, y) \in \R$ such that $\{x, y\} \subseteq S$.
Since $S = \bigcup \I$, there exists $S_1, S_2 \in \I$ such that $\{x, 
y\} \subseteq S_1 \cup S_2$.
In particular, it must be that $(S_1, S_2) \in \R_{\IRR}$ by construction of 
$\AF_{\IRR}$ 
so that $\I$ is not a conflict-free extension of $\AF_{\IRR}$, a contradiction.
Thus, $S \in \ADM(\AF)$ must hold.
Now let $S' \in \IRR(\AF)$ such that $S' \notin \I$.
If such a $S'$ does not exist, the result $S \in \PREF(\AF)$ is clear.
Otherwise, we have two cases.
First, $\{S'\} \notin \CF(\AF_{\IRR})$.
Then $S \cup S'$ is not admissible in $\AF$ by construction of $\AF_{\IRR}$.
Second, $\{S'\} \in \CF(\AF_{\IRR})$.
Since $\I \in \NAIV(\AF_{\IRR})$, we have that $S \cup \{S'\}$ contains some pair of 
$\R_{\IRR}$ and hence that $S \cup S' \notin \ADM(\AF)$ again by construction of 
$\AF_{\IRR}$.
In both cases, we deduce that $S \cup S' \notin \ADM(\AF)$.
As $S \in \ADM(\AF)$, we deduce that $S \in \PREF(\AF)$ as expected.
This concludes the $\supseteq$ part of the proof, and the equality of the lemma holds.

Now we show that this equality defines a bijection between $\PREF(\AF)$ and 
$\NAIV(\AF_{\IRR})$.
As $\IRR(S)$ is uniquely defined for every $S \in \SD(\AF)$, it is sufficient to 
prove 
that for every $\I \in \NAIV(\AF_{\IRR})$, $\I = \IRR(S)$ for some $S \in \PREF(\AF)$.
Let $\I \in \NAIV(\AF)$ and let $S = \bigcup \I$.
By previous discussion $S \in \PREF(\AF)$.
Moreover, $\I \subseteq \IRR(S)$ by definition of $\IRR(S)$.
To prove that $\I = \IRR(S)$, we show that $S' \notin \I$ implies $S' \notin 
\IRR(S)$, for every $S' \in \IRR(\AF)$.
First, observe that if $\I = \IRR(\AF)$, the implication is clear.
Now assume there exists some $S'$ in $\IRR(\AF) \setminus \I$.
We show that $S \cup S' \neq S$.
Because $\I \in \NAIV(\AF)$ and $S = \bigcup \I$, it follows from the construction of 
$\AF_{\IRR}$ that $S \cup S' \notin \ADM(\AF)$.
As $S \in \ADM(\AF)$, we deduce that $S \cup S' \neq S$ must hold. 
Thus, $S' \notin \IRR(S)$ for every $S' \notin \I$.
We obtain $\I = \IRR(S)$, concluding the proof.
\end{proof}

\begin{example}[Running Example]
Recall that $\IRR(\AF) = \{4, 45, 2, 25, 13, 124\}$.
In Figure \ref{fig:ex:construction-2}, we give the argumentation framework 
$\AF_{\IRR}$ 
from Lemma \ref{lem:bijection}.
It is symmetric by construction.
We have $\NAIV(\AF_{\IRR}) = \{\{4, 45\}, \{13\}, \{2, 25\}\}$,
and $\PREF(\AF) = \{45, 13, 25\}$ can be obtained from $\NAIV(\AF_{\IRR})$ by taking the 
union of the irreducible self-defending sets in a given naive set. 
%
\begin{figure}[h!]
	\centering
	\includegraphics[scale=1, page=3]{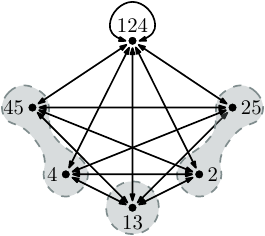}%
	\caption{The argumentation framework $\AF_{\IRR}$ associated to $\AF$.
	Naive sets are highlighted.}
	\label{fig:ex:construction-2}
\end{figure}
\end{example}

We use Lemma \ref{lem:bijection} to find $\PREF(\AF)$.
If $\IRR(\AF)$ is given, $\AF_{\IRR}$ can be computed in 
polynomial time in the size of $\AF$ and $\IRR(\AF)$.
Applying the algorithm of Johnson \textit{et al.} \cite{johnson1988generating}, we derive:

\begin{corollary} \label{cor:poly-delay}
	Let $\AF = \langle \A, \R \rangle$ be an argumentation framework.
	Then, $\PREF(\AF)$ can be computed with polynomial delay and space in the size of 
	$\IRR(\AF)$ and $\AF$.
\end{corollary}

\begin{proof}
	We describe the algorithm.
	First, we construct the argumentation framework $\AF_{\IRR} = \langle \IRR(\AF), 
	\R_{\IRR} \rangle$ in polynomial time in the size of $\IRR(\AF)$ and $\AF$.
	Then, we use the algorithm in \cite{johnson1988generating} to 
	list $\NAIV(\AF_{\IRR})$.
	However, instead of outputting a subset $\I$ of $\IRR(\AF)$, we 
	compute the set $\bigcup \I$.
	Due to Lemma \ref{lem:bijection}, each preferred extension of $\AF$ will be given 
	exactly once.
	As computing $\I$ is done in polynomial time in the size of $\IRR(\AF)$ and $\AF$, 
	the 
	polynomial delay and space of the whole algorithm follows.
\end{proof}

These results based on lattices and implications 
offer a new strategy for computing the preferred extensions of an argumentation framework 
$\AF$: (i) compute $\IS$ \cite{elaroussi:hal-03184819}, (ii) construct $\IRR(\AF)$, and 
(iii) apply Corollary \ref{cor:poly-delay} to find $\PREF(\AF)$.
In general, problem (ii) is open and relates to hypergraph dualization 
\cite{khardon1995translating}.
Yet, there are several classes of closure systems where the task can be conducted in 
polynomial time \cite{beaudou2017algorithms, wild2017joy}, such as distributive and 
meet-semidistributive closure-systems, or 
closure systems associated to symmetric argumentation frameworks (being Boolean).

\section{Conclusion and further works}

In this paper we studied the problem of listing the preferred extensions of an 
argumentation framework, being a hard problem \cite{dimopoulos1996graph}.
We identified cases where this task can be reduced to the enumeration of the naive sets 
of a possibly different framework. 
As this latter problem can be solved with polynomial delay and space 
\cite{johnson1988generating}, our contribution sheds the light on new properties making 
the enumeration of preferred extensions tractable.
These properties pave the way for future research, which we will briefly discuss now.

Being naive-bijective is perhaps the most direct way to relate preferred extensions to 
naive sets.
Even though this property is hard to recognize in general, our characterization provides 
new insights into the structure of naive-bijective frameworks.
For future work it will be interesting to study this characterization in conjunction with 
other properties from particular classes of argumentation frameworks.
By doing so, one may be able to identify further cases where the problem \csmc{NAF} can 
be solved in polynomial time, as we did with bounded in-degree argumentation frameworks.

As for admissible-closed argumentation frameworks, we showed that they are 
naive-recasting. 
In fact, an admissible-closed framework $\AF$ can even be recasted in a framework $\AF'$ 
on the same set of arguments.
This implies that $\AF'$ has polynomial size with respect to $\AF$.
Thus, being able to compute $\AF'$ in polynomial time yields another polynomial delay and 
space algorithm to list the preferred extensions of $\F$.
As such, the open question of Section \ref{sec:admissible-closed} about the complexity of 
finding $\AF'$ is a particularly important part of continuing studies.
Identifying cases where \csmc{ACAF} can be solved in polynomial time, much as \csmc{NAF} 
for naive-bijective frameworks, is another direction for future investigation.


At last, irreducible self-defending sets are a new point of view on the structure of 
argumentation frameworks.
The results from Section \ref{sec:irreducible} show that a small portion of 
self-defending sets are sufficient to efficiently compute the preferred extensions of a 
given framework.
Thus, these results suggest to study argumentation frameworks not only from their graph 
perspective, but also from the lattice and implicational point of view.
Specifically, finding properties of lattices and argumentation frameworks where 
irreducible sets can be computed efficiently is an intriguing question for future work. 

\paragraph{Acknowledgements} 
We are grateful to reviewers along with Oliveiro Nardi, Michael Bernreiter and 
Matthias König for their useful remarks, and for pointing us mistakes especially in 
Theorem 3.1 and Proposition 1.
The first and third authors have been funded by the CMEP 
Tassili project: ``Argumentation et jeux: Structures et Algorithmes'', Codes: 
46085QH-21MDU320 PHC, 2021-2023.
The last author has been funded by the CNRS, France, ProFan project. 
This research is also supported by the French government IDEXISITE initiative 
16-IDEX-0001 (CAP 20-25).

\bibliographystyle{alpha}
\bibliography{biblio}

\newcommand{\etalchar}[1]{$^{#1}$}
\begin{thebibliography}{GLMW20}

\bibitem[AGP19]{alfano2019scaling}
Gianvincenzo Alfano, Sergio Greco, and Francesco Parisi.
\newblock On scaling the enumeration of the preferred extensions of abstract
  argumentation frameworks.
\newblock In {\em Proceedings of the 34th ACM/SIGAPP Symposium on Applied
  Computing}, pages 1147--1153, 2019.

\bibitem[Amg12]{amgoud2012stable}
Leila Amgoud.
\newblock Stable semantics in logic-based argumentation.
\newblock In {\em International Conference on Scalable Uncertainty Management},
  pages 58--71. Springer, 2012.

\bibitem[AP09]{amgoud2009using}
Leila Amgoud and Henri Prade.
\newblock Using arguments for making and explaining decisions.
\newblock {\em Artificial Intelligence}, 173(3-4):413--436, 2009.

\bibitem[Apt90]{apt1990logic}
Krzysztof~R Apt.
\newblock Logic programming.
\newblock {\em Handbook of Theoretical Computer Science, Volume B: Formal
  Models and Sematics (B)}, 1990:493--574, 1990.

\bibitem[BCD07]{bench2007argumentation}
Trevor~JM Bench-Capon and Paul~E Dunne.
\newblock Argumentation in artificial intelligence.
\newblock {\em Artificial intelligence}, 171(10-15):619--641, 2007.

\bibitem[BCG11]{baroni2011introduction}
Pietro Baroni, Martin Caminada, and Massimiliano Giacomin.
\newblock An introduction to argumentation semantics.
\newblock {\em The knowledge engineering review}, 26(4):365--410, 2011.

\bibitem[BDKT97]{bondarenko1997abstract}
Andrei Bondarenko, Phan~Minh Dung, Robert~A Kowalski, and Francesca Toni.
\newblock An abstract, argumentation-theoretic approach to default reasoning.
\newblock {\em Artificial intelligence}, 93(1-2):63--101, 1997.

\bibitem[BDL{\etalchar{+}}14]{baumann2014compact}
Ringo Baumann, Wolfgang Dvor{\'a}k, Thomas Linsbichler, Hannes Strass, and
  Stefan Woltran.
\newblock Compact argumentation frameworks.
\newblock In {\em ECAI}, pages 69--74, 2014.

\bibitem[BMN17]{beaudou2017algorithms}
Laurent Beaudou, Arnaud Mary, and Lhouari Nourine.
\newblock Algorithms for k-meet-semidistributive lattices.
\newblock {\em Theoretical Computer Science}, 658:391--398, 2017.

\bibitem[BTV20]{baroni2020acceptability}
Pietro Baroni, Francesca Toni, and Bart Verheij.
\newblock On the acceptability of arguments and its fundamental role in
  nonmonotonic reasoning, logic programming and n-person games: 25 years later.
\newblock {\em Argument Comput.}, 11(1-2):1--14, 2020.

\bibitem[CDG{\etalchar{+}}15]{charwat2015methods}
G{\"u}nther Charwat, Wolfgang Dvo{\v{r}}{\'a}k, Sarah~A Gaggl, Johannes~P
  Wallner, and Stefan Woltran.
\newblock Methods for solving reasoning problems in abstract argumentation--a
  survey.
\newblock {\em Artificial intelligence}, 220:28--63, 2015.

\bibitem[CMDM05]{coste2005symmetric}
Sylvie Coste-Marquis, Caroline Devred, and Pierre Marquis.
\newblock Symmetric argumentation frameworks.
\newblock In {\em European Conference on Symbolic and Quantitative Approaches
  to Reasoning and Uncertainty}, pages 317--328. Springer, 2005.

\bibitem[CT16]{cocarascu2016argumentation}
Oana Cocarascu and Francesca Toni.
\newblock Argumentation for machine learning: A survey.
\newblock In {\em COMMA}, pages 219--230, 2016.

\bibitem[CVG18]{cerutti2018impact}
Federico Cerutti, Mauro Vallati, and Massimiliano Giacomin.
\newblock On the impact of configuration on abstract argumentation automated
  reasoning.
\newblock {\em International Journal of Approximate Reasoning}, 92:120--138,
  2018.

\bibitem[CW]{caminada2011limitations}
Martin Caminada and Yining Wu.
\newblock On the limitations of abstract argumentation.

\bibitem[DBC01]{dunne2001complexity}
Paul~E Dunne and Trevor~JM Bench-Capon.
\newblock Complexity and combinatorial properties of argument systems.
\newblock {\em University of Liverpool, Department of Computer Science (ULCS),
  Technical report}, 2001.

\bibitem[DD17]{dvovrak2017computational}
Wolfgang Dvo{\v{r}}{\'a}k and Paul~E Dunne.
\newblock Computational problems in formal argumentation and their complexity.
\newblock {\em Journal of Logics and their Applications}, 4(8):2557--2622,
  2017.

\bibitem[DDLW15]{dunne2015characteristics}
Paul~E Dunne, Wolfgang Dvo{\v{r}}{\'a}k, Thomas Linsbichler, and Stefan
  Woltran.
\newblock Characteristics of multiple viewpoints in abstract argumentation.
\newblock {\em Artificial Intelligence}, 228:153--178, 2015.

\bibitem[DF12]{doignon2012knowledge}
Jean-Paul Doignon and Jean-Claude Falmagne.
\newblock {\em Knowledge Spaces}.
\newblock {Springer Science \& Business Media}, 2012.

\bibitem[DSLW16]{dunne2016investigating}
Paul~E Dunne, Christof Spanring, Thomas Linsbichler, and Stefan Woltran.
\newblock Investigating the relationship between argumentation semantics via
  signatures.
\newblock In {\em Proceedings of the 39th Annual German Conference on AI
  (KI’16)}, pages 271--277. Springer, 2016.

\bibitem[DT96]{dimopoulos1996graph}
Yannis Dimopoulos and Alberto Torres.
\newblock Graph theoretical structures in logic programs and default theories.
\newblock {\em Theoretical Computer Science}, 170(1-2):209--244, 1996.

\bibitem[Dun95]{dung1995acceptability}
Phan~Minh Dung.
\newblock On the acceptability of arguments and its fundamental role in
  nonmonotonic reasoning, logic programming and n-person games.
\newblock {\em Artificial intelligence}, 77(2):321--357, 1995.

\bibitem[Dun07]{dunne2007computational}
Paul~E Dunne.
\newblock Computational properties of argument systems satisfying
  graph-theoretic constraints.
\newblock {\em Artificial Intelligence}, 171(10-15):701--729, 2007.

\bibitem[ENR21]{elaroussi:hal-03184819}
Mohammed Elaroussi, Lhouari Nourine, and Mohammed~Said Radjef.
\newblock {Lattice point of view for argumentation framework}.
\newblock working paper or preprint, March 2021.

\bibitem[Gan10]{ganter2010two}
Bernhard Ganter.
\newblock Two basic algorithms in concept analysis.
\newblock In {\em International conference on formal concept analysis}, pages
  312--340. Springer, 2010.

\bibitem[GLMW20]{gaggl2020design}
Sarah~A Gaggl, Thomas Linsbichler, Marco Maratea, and Stefan Woltran.
\newblock Design and results of the second international competition on
  computational models of argumentation.
\newblock {\em Artificial Intelligence}, 279:103193, 2020.

\bibitem[Gr{\"a}11]{gratzer2011lattice}
George Gr{\"a}tzer.
\newblock {\em Lattice theory: foundation}.
\newblock Springer Science \& Business Media, 2011.

\bibitem[JYP88]{johnson1988generating}
David~S Johnson, Mihalis Yannakakis, and Christos~H Papadimitriou.
\newblock On generating all maximal independent sets.
\newblock {\em Information Processing Letters}, 27(3):119--123, 1988.

\bibitem[Kha95]{khardon1995translating}
Roni Khardon.
\newblock Translating between {{Horn}} representations and their characteristic
  models.
\newblock {\em Journal of Artificial Intelligence Research}, 3:349--372, 1995.

\bibitem[KPW17]{kroll2017complexity}
Markus Kr{\"o}ll, Reinhard Pichler, and Stefan Woltran.
\newblock On the complexity of enumerating the extensions of abstract
  argumentation frameworks.
\newblock In {\em Proceedings of the 26th International Joint Conference on
  Artificial Intelligence}, pages 1145--1152. AAAI Press, 2017.

\bibitem[Mar75]{markowsky1975factorization}
George Markowsky.
\newblock The factorization and representation of lattices.
\newblock {\em Transactions of the American Mathematical Society},
  203:185--200, 1975.

\bibitem[MP09]{mcburney2009dialogue}
Peter McBurney and Simon Parsons.
\newblock Dialogue games for agent argumentation.
\newblock In {\em Argumentation in artificial intelligence}, pages 261--280.
  Springer, 2009.

\bibitem[Pol87]{pollock1987defeasible}
John~L Pollock.
\newblock Defeasible reasoning.
\newblock {\em Cognitive science}, 11(4):481--518, 1987.

\bibitem[P{\"u}h20]{puhrer2020realizability}
J{\"o}rg P{\"u}hrer.
\newblock Realizability of three-valued semantics for abstract dialectical
  frameworks.
\newblock {\em Artificial Intelligence}, 278:103198, 2020.

\bibitem[RS09]{rahwan2009argumentation}
Iyad Rahwan and Guillermo~R Simari.
\newblock {\em Argumentation in artificial intelligence}, volume~47.
\newblock Springer, 2009.

\bibitem[VNM47]{von1947theory}
John Von~Neumann and Oskar Morgenstern.
\newblock Theory of games and economic behavior, 2nd rev.
\newblock 1947.

\bibitem[Wil17]{wild2017joy}
Marcel Wild.
\newblock The joy of implications, aka pure horn formulas: mainly a survey.
\newblock {\em Theoretical Computer Science}, 658:264--292, 2017.

\bibitem[XC18]{xu2018initial}
Yuming Xu and Claudette Cayrol.
\newblock Initial sets in abstract argumentation frameworks.
\newblock {\em Journal of Applied Non-Classical Logics}, 28(2-3):260--279,
  2018.

\end{thebibliography}

\end{document}